\documentclass{amsart}
\usepackage{graphicx,color}
\usepackage{amsfonts,amssymb,amscd,amsmath,euscript,enumerate,verbatim,calc}
 
\usepackage[pagewise]{lineno}
\usepackage{rotating}

\newtheorem{construction}{Construction}
\usepackage{multirow}
\usepackage[table,xcdraw]{xcolor}

\newtheorem{theorem}{Theorem}[section]
\newtheorem{corollary}{Corollary}

\newtheorem{proposition}{Proposition}

\theoremstyle{definition}
\newtheorem{definition}[theorem]{Definition}
\newtheorem{remark}{Remark}

\email{tbardini@uottawa.ca}
\email{lmoura@uottawa.ca}

\thanks{Thais Bardini Idalino acknowledges funding granted from CNPq-Brazil [233697/2014-4]. Lucia Moura was supported by an NSERC discovery grant.}


\begin{document}
\title{Embedding cover-free families and cryptographical applications}

\maketitle

\centerline{Thais Bardini Idalino \and Lucia Moura}
\medskip
{\footnotesize

	\centerline{University of Ottawa}
}

\begin{abstract}
Cover-free families are set systems used as solutions for a large variety of problems, and in particular, problems where we deal with $n$ elements and want to identify $d$ invalid ones among them by performing only $t$ tests ($t \leq n$). 
We are specially interested in cryptographic problems, and we note that some of these problems need cover-free families with an increasing size $n$. Solutions that propose the increase of $n$, such as \emph{monotone families} and \emph{nested families}, have been recently considered in the literature. In this paper, we propose a generalization that we call \emph{embedding families}, which allows us to increase both $n$ and $d$. We propose constructions of \emph{embedding families} using polynomials over finite fields, and show specific cases where this construction allows us to prioritize increase of $d$ or $n$ with good compression ratios.
We also provide new constructions for monotone families with improved compression ratio. Finally, we show how to use embedded sequences of orthogonal arrays and packing arrays to build embedding families.
\end{abstract}

\section{Introduction}

A cover-free family (CFF) is a set system usually studied in the context of group testing applications. In this scenario, we are given a set of $n$ elements and want to identify up to $d$ invalid ones in a more efficient way than testing each one of them individually. A $d$-cover-free family (or $d$-CFF) will indicate how to group the $n$ elements into $t$ groups ($t \leq n$), and by performing only $t$ tests we will be able to identify up to $d$ invalid elements.
These families are used to solve several problems in cryptography, such as one-time and multiple-times digital signature schemes~\cite{ZaveruchaSignature,hors++}, fault-tolerant aggregation of signatures~\cite{hartung,nested,zaverucha}, modification localization on signed documents and redactable signatures~\cite{thaisIPL}, broadcast authentication~\cite{broadcastauth}, broadcast encryption~\cite{broadcastenc}, among others~\cite{CFFapp}.

We can represent $d$-CFFs as set systems or their corresponding incidence matrices.
%
A set system $\mathcal{F} = (X, \mathcal{B})$ consists of a set $X = \{x_1, \ldots, x_t\}$ with $|X| = t$, and a collection $\mathcal{B} = \{B_1, \ldots, B_n\}$ with $B_i \subseteq X, 1 \leq i \leq n,$ and $|\mathcal{B}| = n$. A $d$-cover-free family, denoted $d-$CFF$(t,n)$, is a set system such that for any subset $B_{i_0} \in \mathcal{B}$ and any other $d$ subsets $B_{i_1}, \ldots, B_{i_d} \in \mathcal{B}$, we have
\begin{equation}\label{property:cff}
B_{i_0} \nsubseteq \bigcup_{j=1}^{d}B_{i_j}.
\end{equation}	
Family $\mathcal{F}$ can be represented by its $t \times n$ binary incidence matrix $\mathcal{M}$:
\[\mathcal{M}_{i,j} =
\begin{cases}
1, & \text{ if } x_i \in B_j, \\
0 & \text{ otherwise.}
\end{cases}\]
In the remaining of this paper, we may use the term $d$-CFFs to refer to their incidence matrices. 
For a basic reference in combinatorial group testing see \cite{gtLivro}, and for more information about combinatorial designs see \cite{handbook}.
Table \ref{2cff} shows how a $2-$CFF($9,12$) can be used to test $n=12$ elements and identify up to $d=2$ invalid ones using $t=9$ tests.

\begingroup
\setlength{\tabcolsep}{8pt} 
\begin{table}[h!]\label{2cff}
	\caption{Example of a $2$-CFF($9,12$) used in group testing.}
	\centering
	\begin{tabular}{l|c c c c c c c c c c c c |l}
		& 1 & 2 & 3 & 4 & 5 & 6 & 7 & 8 &9 &10& 11& 12& result\\ \hline
		$test_1$  & 1& 0 & 0 & 1  & 0 & 0 &  1&0 &0 & 1 &0 & 0 & 0\\
		$test_2$  &  1 & 0 & 0 & 0 &1  & 0 & 0 &  1 & 0 &  0 & 1&  0 & 0\\
		$test_3$  &  1 & 0 & 0 & 0 & 0 &  1 & 0 &0 & 1 &0 & 0 & 1 &\textbf{1}\\
		$test_4$  & 0 &  1& 0 & 1 & 0 & 0 & 0 &0 & 1 &0 & 1 & 0 &0\\
		$test_5$  & 0 & 1 & 0 & 0 & 1 & 0 & 1 & 0 & 0 & 0 &  0 & 1&\textbf{1}\\
		$test_6$  & 0 & 1 & 0 & 0 & 0 & 1 & 0 &1& 0 & 1 & 0 &  0& 0\\
		$test_7$  & 0 & 0 &  1 &  1 & 0 & 0 & 0 &  1 & 0 & 0 & 0 &  1 & \textbf{1}\\
		$test_8$  & 0 & 0 &  1 & 0 &  1 & 0 & 0 & 0 &  1&  1 & 0 & 0 & \textbf{1}\\
		$test_9$  & 0 & 0 &  1 & 0 & 0 & 1 & 1 &0 & 0 & 0 & 1 & 0  & \textbf{1}\
	\end{tabular}\\
\end{table}
\endgroup
The columns of the matrix represent the elements to be tested, and the rows indicate which elements we are testing together.
After performing the $9$ testes we obtain the last column with some results. If all the elements in a test are valid, the test \emph{passes} (represented as $0$), but if there is at least one defective element in a test, it \emph{fails} (represented as $1$). By the tests that pass we can identify all the valid elements, which in the example are $1,2,4,5,6,7,8,9,10,$ and $11$. 
Since the remaining set of elements $S = \{3,12\}$ have $|S| \leq d$, by the definition of $d$-CFF we can conclude $3$ and $12$ are the defectives (since each of them are the only possible cause for failure in tests $test_3, test_5, test_7, test_8, test_9$).

CFFs provide a practical solution for problems where the number of elements $n$ is known a priori. For applications where $n$ is not known or can dynamically increase over time, we need a scheme that provides matrix growth. This can be done with a special sequence of $d$-CFFs, where the previous matrix is a sub-matrix of the next ones, so that we can reuse the groups and computations we already performed for smaller values of $n$.   \emph{Monotone families}~\cite{hartung} and \emph{nested families}~\cite{nested} of $d$-CFFs are examples of such special sequences that are used to acchieve unbounded fault-tolerant aggregate digital signatures. 
One drawback of these families is that $d$ must be constant, so we need more general sequences of families if we wish $d$ to grow with $n$.  For this purpose, in this paper,
we define a generalization of both monotone and nested families called  \emph{embedding families}.

To compare the efficiency of different families of CFF, we consider the \emph{compression ratio}, which is given by $\rho(n)$ when $\frac{n}{t(n)}$ is $O(\rho(n))$. The compression ratio measures the efficiency gained from group testing, which performs $t(n)$ tests rather than $n$. We look for constructions with $\rho(n)$ as large as possible, for example, the ones that meet or are close to the information theoretical bound $\rho(n) = \frac{n}{(d^2/\log d) \log n}$~\cite{furedi}. 
In the literature, monotone families with constant compression ratio have been given in~\cite{hartung}, while several constructions of (the more general) nested families 
with compression ratio closer to the information theoretical bound were presented in~\cite{nested}.
Considering the limitation of constant $d$ for monotone and nested families, the present paper gives constructions of embedding families with good compression ratios that allows $d$ 
to grow with $n$.

Our contributions in this paper are as follows. 
We revisit a construction of $d$-CFF by Erd\"os et al.~\cite{erdospoly} based on polynomials over finite fields (Theorem~\ref{theopoly}) and highlight some useful properties 
related to progressive $d$ (Theorem~\ref{smallercff}). This property can be observed in the example in Table \ref{table:2cff}, where a matrix has submatrices with smaller $d$ inside it, which allows us to early abort the testing after enough tests are done for the actual level of defectives.
We then give a general construction of embedding families of CFF, each CFF based on this polynomial construction, and stacked together using extension fields (Theorem~\ref{theofamily}). 
Specific applications of this general construction  give embedding families with sublinear $d = d(n)$ and $\rho(n) = n^{1-\frac{2}{k+1}}$ (Corollary \ref{dinc}) as well as with constant $d$ and $\rho(n)=\frac{n}{\log n}$, achieving the information theoretical upper bound  (Corollary \ref{kinc}). Moreover, we show it is possible to adapt this construction to build monotone families with compression ratio $\rho(n) = n^{1-\frac{1}{k+1}}$, for each arbitrary constant $k \geq 1$ (Theorem~\ref{monotoneinc}), which is much superior to the constant compression ratios obtained in~\cite{hartung}. Finally, we show that families of orthogonal arrays and packing arrays with some specific properties generalize the polynomial construction of embedding families (Proposition~\ref{embeddingPA}), which can open the door for new constructions in the future.


In Section~\ref{embed}, we define embedding families and discuss cryptographical applications. In Section~\ref{embeddingsec}, we give constructions of embedding families based on polynomials over finite fields. In Section~\ref{appsec}, we discuss the use of these constructions in applications, and challenges related to drop of actual compression ratios when columns are not used.
In Section~\ref{section:generalized}, we generalize the polynomial construction by using other combinatorial designs, and in Section~\ref{sec:conclusion}, we give conclusions.

\section{Embedding Sequences and its Applications}\label{embed}

In this section, we present CFF constructions for unbounded applications, which are applications where $n$ may not be known a priori or can grow over time. 
We introduce the notion of \emph{embedding families} to be a sequence of CFFs that allows for the increase of $n$ and $d$, and we also show how they are a generalization of nested families \cite{nested} and monotone families \cite{hartung}. 




\begin{definition}[Embedding family]\label{embeddingfamily} Let $d(l)$ be a positive integer and let $(\mathcal{M}^{(l)})_l$ be a sequence of incidence matrices of cover-free families $(\mathcal{F}_l)_l = (X_l, \mathcal{B}_l)_l$, where $\mathcal{M}^{(l)}$ is a $d(l)$-CFF with number of rows and columns denoted by rows($l$) and cols($l$), respectively. $(\mathcal{M}^{(l)})_l$ is a \emph{embedding family} of incidence matrices of CFFs, if $X_l \subseteq X_{l+1}$, rows($l$) $\leq$ rows($l+1$), and cols($l$) $\leq$ cols($l+1$), $d(l) \leq d(l+1)$ and
		\[\mathcal{M}^{(l+1)}= \begin{pmatrix} \mathcal{M}^{(l)} & Y\\ Z & W \end{pmatrix}.\]
\end{definition}	

We can see that monotone and nested families are a special case of embedding families. They allow us to increase $n$ for fixed $d$, with a $Z$ that has a special format. The definitions are shown below.

\begin{definition}[Nested family]
	A nested family of incidence matrices of $d$-CFFs is an embedding family of incidence matrices with fixed $d$ such that each row of $Z$ is one of the rows of $\mathcal{M}^{(l)}$, a row of all zeros, or a row of all ones.
\end{definition}	

Nested families were defined in \cite{nested} to solve a problem in unbounded aggregation of digital signatures, where three different constructions with increasing compression ratio are presented.


\begin{definition}[Monotone family] 
	A monotone family of incidence matrices of $d$-CFFs is an embedding family of incidence matrices  with fixed $d$ such that $Z$ is a matrix of zeros.
%
\end{definition}	
Monotone families were introduced by Hartung et al. \cite{hartung} to solve the problem of unbounded aggregation of signatures. They showed a concrete instantiation of monotone families with a constant compression ratio. We show in Theorem \ref{monotoneinc} a construction for monotone families with $\rho(n) = n^{1 - \frac{1}{c}}$, for a constant c.

\subsection{Cryptographical Applications}
We can think of a variety of applications for embedding families. General group testing applications, for example, may take advantage of this family for cases where increasing $n$ is necessary, together with the possibility of larger $d$'s. Here we are most interested in applications related to cryptography.

\noindent
\textbf{Aggregation of signatures:}
The purpose of aggregation of signatures is to save on storage, communication and verification time by combining several signatures together \cite{Boneh}, and $d$-CFFs are known to provide this while allowing the identification of up to $d$ invalid signatures \cite{hartung,nested,zaverucha}. Since the number of signatures may not be known a priori, it is important to have a $d$-CFF that allows the increase of $n$. However, after signatures are aggregated together using a smaller matrix, the individual signatures are discarded and we only keep the aggregated ones, which implies that larger matrices should not require the knowledge of those signatures that were discarded. A solution for this was first proposed by Hartung et al.  \cite{hartung} using monotone families, where the zero matrix bellow $\mathcal{M}^{(l)}$ address this problem directly. Nested families can also be used as a solution to this problem, since its submatrix $Z$ also address this problem by requiring only one extra aggregation of past signatures \cite{nested}. The advantage of nested families is that the known constructions present a much better compression ratio, which is closer to the information theoretical bound, and consequently gives smaller aggregate signature size \cite{nested}. 

\noindent
\textbf{Broadcast encryption:}
In this scheme, a sender broadcasts encrypted messages to a set of $n$ users, but want to prevent some of them from recovering these messages.
An example of such application is paid television, where only some users can have access to certain paid channels \cite{broadcastenc}. 
Gafni et al. \cite{broadcastenc} propose the use of $d$-CFF for distributing the keys that are used to encrypt and decrypt the message. In this scenario, the columns of the $d$-CFF represent the users, and the rows represent a set of $t$ keys. Each user receives a subset of the keys according to their column, and the $d$-CFF property  guarantees that we can remove up to $d$ users and their respective keys without compromising the ability of the remaining users to decrypt the content. In this scheme, an embedding family would provide a \emph{fully scalable} scheme \cite{broadcastenc}, where we can add new users by increasing $n$, and additionally handle a larger number $d$ of users that may be removed.



\noindent
\textbf{Broadcast authentication:} In this scheme, sender and receivers agree on secret keys, and these keys are used to guarantee the authenticity of broadcasted messages. 
However, there may be malicious users who can get together and use their secret keys and previous communication to create fraudulent messages, which may be accepted by some users as authentic  \cite{CFFapp,broadcastauth}.
Safavi-Naini and Wang \cite{broadcastauth} propose the use of $d$-CFFs to manage the distribution of keys. Again, the columns of the CFF represent the users, the rows represent the keys, and each user receives a subset of these keys corresponding to their column of the matrix. Because of the $d$-CFF, the union of the keys of up to $d$ malicious users is not enough to create a fraudulent message~\cite{CFFapp,broadcastauth}.  In this scenario, we could again think of embedding families as a way to provide an increase in the number of receivers $n$ and malicious users $d$ that the system can handle.



\section{Embedding Sequences Using Polynomials Over Finite Fields}\label{embeddingsec}

In this section, we present a construction for embedding sequences based on a known construction of $d$-CFFs.
We start by presenting a construction proposed by Erd{\"o}s, Frankl and F{\"u}redi \cite{erdospoly}, that uses polynomials of degree up to $k$ over a finite field $\mathbb{F}_q$, denoted as $f \in \mathbb{F}_q[x]_{\leq k}$, and generates a $d$-CFF($t=q^2, n=q^{k+1}$) for $d \leq \frac{q-1}{k}$. We also note that this known construction presents some interesting properties that allow us to ignore a few rows of the $d$-CFF if we need smaller values of $d$ (see Theorem \ref{smallercff}). We finally show how to use this polynomial construction to obtain embedding families, and how we can focus on prioritizing increases of $d$, $n$, and obtain monotone families with increasing compression ratio from it.

\begin{construction}[Erd{\"o}s et al. \cite{erdospoly}]\label{constpoly}
Let $q$ be a prime power, $k$ a positive integer, and consider the elements of the finite field as $\mathbb{F}_q = \{x_1, \ldots, x_q\}$. We define ($X, \mathcal{B}$) as follows, for each polynomial $f \in \mathbb{F}_q[x]_{\leq k}$.

\[X = \mathbb{F}_q \times \mathbb{F}_q = \{(x_i,x_j): i,j = 1, \ldots, q\},\]
\[B_f = \{(x_1, f(x_1)), \ldots, (x_q, f(x_q))\} \subset X,\]
\[\mathcal{B} = \{B_f: f \in \mathbb{F}_q[x]_{\leq k}\}.\] 

Call $C_{q,k}$ the incidence matrix obtained.
\end{construction}

The argument in the following proof was observed in \cite{hartung,hors++}.
\begin{theorem}\label{theopoly}[Erd{\"o}s, Frankl, F{\"u}redi (1985)]
Let $q$ be a prime power, $k\geq 1$, and $d \leq \frac{q-1}{k}$. Then $C_{q,k}$ from Construction \ref{constpoly} is a $d$-CFF($q^2, q^{k+1}$).
\end{theorem}
\begin{proof}
It is easy to see that $|X| = q^2, |B_f| = q,$ and $|\mathcal{B}| = q^{k+1}$. Moreover, for $f_i \neq f_j$ we have $|B_{f_i} \cap B_{f_j}| \leq k$. So when we consider any $d+1$ subsets $B_{f_0}, B_{f_1}, \ldots, B_{f_d}$, we have that 
\[\bigg| B_{f_0} \Big\backslash \bigcup_{i=1}^{d}B_{f_i}\bigg| \geq q- dk \geq 1,\]

\noindent the last inequality due to hypothesis  $dk \leq  q-1$. Thus, $B_{f_0} \not\subseteq \bigcup_{i=1}^{d}B_{f_i}$.
Therefore, ($X, \mathcal{B}$) represents a $d$-CFF($q^2, q^{k+1}$) with $d \leq \frac{q-1}{k}$. 
\end{proof}

As an example, for $q=3, d=2, k=1$, we have $X = \mathbb{F}_3 \times \mathbb{F}_3 = \mathbb{Z}_3 \times \mathbb{Z}_3$, polynomials with coefficients in $\mathbb{F}_3 $ of degree up to $k=1$, and obtain the $2$-CFF($9,9$) in Table~\ref{table:2cff}.

\begin{table}[h!]
	\caption{Example of a $1$-CFF($6,9$) and a $2$-CFF($9,9$).}
	\label{table:2cff}
	\begin{tabular}{c*{9}{>{\centering\arraybackslash}p{.07\linewidth}}}
		& $0$ & $1$ & $2$ & $x$ & $x+1$ & $x+2$ & $2x$ & $2x+1$ & $2x+2$ \\ \cline{2-10}
		\multicolumn{1}{l|}{$(0,0)$} & $1$ & $0$ & $0$ & $1$ & $0$   & $0$   & $1$  & $0$    & \multicolumn{1}{c|}{$0$}    \\ 
		\multicolumn{1}{l|}{$(0,1)$} & $0$ & $1$ & $0$ & $0$ & $1$   & $0$   & $0$  & $1$    & \multicolumn{1}{c|}{$0$}    \\ 
		\multicolumn{1}{l|}{$(0,2)$} & $0$ & $0$ & $1$ & $0$ & $0$   & $1$   & $0$  & $0$    & \multicolumn{1}{c|}{$1$}    \\ 
		\multicolumn{1}{l|}{$(1,0)$} & $1$ & $0$ & $0$ & $0$ & $0$   & $1$   & $0$  & $1$    & \multicolumn{1}{c|}{$0$}    \\ 
		\multicolumn{1}{l|}{$(1,1)$} & $0$ & $1$ & $0$ & $1$ & $0$   & $0$   & $0$  & $0$    & \multicolumn{1}{c|}{$1$}    \\ 
		\multicolumn{1}{l|}{$(1,2)$} & $0$ & $0$ & $1$ & $0$ & $1$   & $0$   & $1$  & $0$    & \multicolumn{1}{c|}{$0$}    \\ \cline{2-10}
		\multicolumn{1}{l|}{$(2,0)$} & $1$ & $0$ & $0$ & $0$ & $1$   & $0$   & $0$  & $0$    & \multicolumn{1}{c|}{$1$}    \\ 
		\multicolumn{1}{l|}{$(2,1)$} & $0$ & $1$ & $0$ & $0$ & $0$   & $1$   & $1$  & $0$    & \multicolumn{1}{c|}{$0$}    \\ 
		\multicolumn{1}{l|}{$(2,2)$} & $0$ & $0$ & $1$ & $1$ & $0$   & $0$   & $0$  & $1$    & \multicolumn{1}{c|}{$0$}    \\ \cline{2-10}
	\end{tabular}
\end{table}

We note that the $d$-CFF construction presented above has a special structure that can be explored to guarantee some interesting properties. 
Here we focus on the property of being able to discard some rows of the $d$-CFF incidence matrix when smaller values of $d$ are enough, and the ability to increase $d$ as necessary (up to a maximum) by considering extra rows.
This property is important because it allows the testing algorithm to do an early abort doing only enough tests to detect the actual defectives $d'$ if $d'<d$, where $d$ is the maximum $\lfloor \frac{q-1}{k} \rfloor$ allowed by the $d$-CFF matrix. In the example above, for instance, if we discard the last three rows we obtain a $1$-CFF($6,9$).

For the remaining of this paper we consider a \emph{block of rows} in this construction as the set of $q$ rows $\{(x_i, x_0),  (x_i, x_1), \ldots, (x_i, x_q) \}$ for every $x_i \in \mathbb{F}_q$. When we restrict our matrix to $i$ blocks of rows, we are considering $X = \{x_1, \ldots, x_i\} \times \mathbb{F}_q$, $B_f(i) = \{(x_1, f(x_1)), \ldots, (x_i, f(x_i))\}$ and $\mathcal{B}(i) = \{B_f(i): f \in \mathbb{F}_q[x]_{\leq k}\}$.


\begin{construction}\label{constIncd}
Let $q$ be a prime power, $k\geq 1$, and $q \geq dk+1$. Let $C_{q,k,d}$ be the matrix corresponding to $\mathcal{B}(dk+1)$, or in other words, the matrix $C_{q,k}$ from Construction \ref{constpoly} restricted to the first $(dk+1)$ blocks of rows.
\end{construction}

\begin{theorem}\label{smallercff}
	Let $q$ a prime power, $k\geq 1$, and $q \geq dk+1$. Then $C_{q,k,d}$ from Construction \ref{constIncd} is a $d$-CFF($(dk+1)q, q^{k+1}$).
\end{theorem}
\begin{proof}
	The proof follows a similar argument as for Theorem \ref{theopoly}. We have $|B_{f_i}(a) \cap B_{f_j}(a)| \leq k$ for all $f_i,f_j \in \mathbb{F}_q[x]_{\leq k}$ and $1 \leq a \leq q$. Taking any $d+1$ distinct sets $B_{f_0}(dk+1), B_{f_1}(dk+1), \ldots, B_{f_d}(dk+1)$, we have $|B_{f_0}(dk+1) \backslash \bigcup_{i=1}^{d}B_{f_i}(dk+1)| \geq dk+1- dk \geq 1$. So $\mathcal{B}(dk+1)$ has the $d$-cover-free property and $C_{q,k,d}$ is a $d$-CFF($(dk+1)q, q^{k+1}$).
%
\end{proof}

For the case of $k=1$, we observe that this incremental $d$ property was given in \cite{PPS}. This is because the constructions presented in ~\cite{PPS} are based on Mutually Orthogonal Latin Squares (MOLS), which can be constructed with polynomials of degree $k=1$.

\subsection{Embedding Sequence Construction}\label{embeddingconstsec}

%

In this section, we give constructions of embedding sequences of CFFs using the previous construction and extension fields.
We start with a prime power $q$ and consider the increase as $q^{2^i}$ for $i \geq 0$, which gives a direct increase of $n$ and $t$.
Since we are increasing $q$, we may also consider to increase $k$ and/or $d$ as long as we respect the inequality $q \geq dk +1$. By increasing $k$ to some $k'$ we make $n$ grow faster and consequently improve the compression ratio. By increasing $d$ to $d'$ we can allow the identification of more defective elements, which may be necessary as the number of elements $n$ grows.

The following theorem is the basic step to be used in the embedding sequence construction.

\begin{theorem}\label{theoinc}
Let $q\geq dk +1, k' \geq k, d' \geq d$ and $q^2 \geq d'k' +1$. Let $C_{q,k,d}$ and $C_{q^2,k',d'}$ be the CFF matrices obtained from the polynomial construction (Construction \ref{constIncd}). Then, there exists $\mathbf{C}_{q^2,k',d'}$ obtained from $C_{q^2,k',d'}$ by a column and row permutation that has the form

	\[\mathbf{C}_{q^2,k',d'}= \begin{pmatrix} C_{q,k,d} & Y\\ Z & W \end{pmatrix}.\]
	
\noindent	Moreover, $\mathbf{C}_{q^2,k',d'}$ is a $d'$-CFF($(d'k'+1)q^2, q^{2(k'+1)}$) and $C_{q,k,d}$ is a $d$-CFF($(dk+1)q, q^{(k+1)}$).
\end{theorem}
\begin{proof}
	To form $\mathbf{C}_{q^2,k',d'}$ we first list the rows of $C_{q^2,k',d'}$ that are of the form $(x_i,x_j)$ for all $1 \leq i \leq dk+1$, $1 \leq j \leq q$ and its columns indexed by $B_f$ for all $f \in \mathbb{F}_q[x]_{\leq k}$, followed by the remaining rows and columns in some order. Since $\mathbb{F}_q$ is a subfield of $\mathbb{F}_{q^2}$, $\mathbb{F}_q[x]_{\leq k} \subseteq \mathbb{F}_{q^2}[x]_{\leq k'}$, so we can list the columns starting by all $f \in \mathbb{F}_q[x]_{\leq k}$ followed by $f \in \mathbb{F}_{q^2}[x]_{\leq k'} \setminus \mathbb{F}_q[x]_{\leq k}$, and the evaluations of polynomials in $\mathbb{F}_q$ and $\mathbb{F}_{q^2}$ give the same result. 
	Thus, the $(dk+1)q \times q^{k+1}$ submatrix of $\mathbf{C}_{q^2,k',d'}$ in the upper left corner coincides precisely with $C_{q,k,d}$. The fact they are $d$-CFF and $d'$-CFF comes from Theorem \ref{smallercff}.
\end{proof}	

As an example, consider $q=9, d=2$ and $k=1$, and increase the last two parameters to $d'=4, k'=2$. We are able to increase $C_{3,1,2}$ from Table~\ref{table:2cff} and obtain a $4$-CFF($81,729$) $\mathbf{C}_{9,2,4}$, as shown in Table~\ref{table:extensionField}.

\begin{table}[h!]
	\tiny
	\caption{Example of a $4$-CFF($81,729$)}
	\label{table:extensionField}
	\begin{tabular}{ccccccccccclll}
		& $0$                     & $1$                     & $2$                     & $x$                     & $x+1$                   & $x+2$ & $2x$ & $2x+1$ & $2x+2$                   & \multicolumn{1}{l}{$\alpha$} & $\alpha+1$ & $\ldots$ & \begin{tabular}[c]{@{}l@{}}$(2\alpha + 2)x^2 +$ \\ $(2\alpha + 2)x +$ \\ $2\alpha + 2$\end{tabular} \\ \cline{2-14} 
		\multicolumn{1}{c|}{$(0,0)$}                                                               & $1$                     & $0$                     & $0$                     & $1$                     & $0$                     & $0$   & $1$  & $0$    & \multicolumn{1}{c|}{$0$} & \multicolumn{4}{c|}{\multirow{14}{*}{$\Huge{\ldots}$}}                                                                                                 \\
		\multicolumn{1}{c|}{$(0,1)$}                                                               & $0$                     & $1$                     & $0$                     & $0$                     & $1$                     & $0$   & $0$  & $1$    & \multicolumn{1}{c|}{$0$} & \multicolumn{4}{c|}{}                                                                                                                                  \\
		\multicolumn{1}{c|}{$(0,2)$}                                                               & $0$                     & $0$                     & $1$                     & $0$                     & $0$                     & $1$   & $0$  & $0$    & \multicolumn{1}{c|}{$1$} & \multicolumn{4}{c|}{}                                                                                                                                  \\
		\multicolumn{1}{c|}{$(1,0)$}                                                               & $1$                     & $0$                     & $0$                     & $0$                     & $0$                     & $1$   & $0$  & $1$    & \multicolumn{1}{c|}{$0$} & \multicolumn{4}{c|}{}                                                                                                                                  \\
		\multicolumn{1}{c|}{$(1,1)$}                                                               & $0$                     & $1$                     & $0$                     & $1$                     & $0$                     & $0$   & $0$  & $0$    & \multicolumn{1}{c|}{$1$} & \multicolumn{4}{c|}{}                                                                                                                                  \\
		\multicolumn{1}{c|}{$(1,2)$}                                                               & $0$                     & $0$                     & $1$                     & $0$                     & $1$                     & $0$   & $1$  & $0$    & \multicolumn{1}{c|}{$0$} & \multicolumn{4}{c|}{}                                                                                                                                  \\
		\multicolumn{1}{c|}{$(2,0)$}                                                               & $1$                     & $0$                     & $0$                     & $0$                     & $1$                     & $0$   & $0$  & $0$    & \multicolumn{1}{c|}{$1$} & \multicolumn{4}{c|}{}                                                                                                                                  \\
		\multicolumn{1}{c|}{$(2,1)$}                                                               & $0$                     & $1$                     & $0$                     & $0$                     & $0$                     & $1$   & $1$  & $0$    & \multicolumn{1}{c|}{$0$} & \multicolumn{4}{c|}{}                                                                                                                                  \\
		\multicolumn{1}{c|}{$(2,2)$}                                                               & $0$                     & $0$                     & $1$                     & $1$                     & $0$                     & $0$   & $0$  & $1$    & \multicolumn{1}{c|}{$0$} & \multicolumn{4}{c|}{}                                                                                                                                  \\ \cline{2-10}
		\multicolumn{1}{c|}{\multirow{3}{*}{$(\mathbb{F}_3,\mathbb{F}_9 \setminus \mathbb{F}_3)$}} & \multicolumn{9}{c|}{\multirow{3}{*}{\large{0}}}                                                                                                                                            & \multicolumn{4}{c|}{}                                                                                                                                  \\
		\multicolumn{1}{c|}{}                                                                      & \multicolumn{9}{c|}{}                                                                                                                                                              & \multicolumn{4}{c|}{}                                                                                                                                  \\
		\multicolumn{1}{c|}{}                                                                      & \multicolumn{9}{c|}{}                                                                                                                                                              & \multicolumn{4}{c|}{}                                                                                                                                  \\ \cline{2-10}
		\multicolumn{1}{c|}{$\vdots$}                                                              & \multicolumn{9}{c}{$\ldots$}                                                                                                                                                      & \multicolumn{4}{c|}{}                                                                                                                                  \\
		\multicolumn{1}{c|}{$(2\alpha+2,2\alpha+2)$}                                               & \multicolumn{1}{l}{$0$} & \multicolumn{1}{l}{$0$} & \multicolumn{1}{l}{$0$} & \multicolumn{1}{l}{$1$} & \multicolumn{1}{l}{$0$} & \multicolumn{4}{c}{$\ldots$}                     & \multicolumn{4}{c|}{}                                                                                                                                  \\ \cline{2-14} 
	\end{tabular}
\end{table}

Theorem \ref{theoinc} allows the construction of infinite families of CFF with special properties, namely monotone families, and embedding families with increasing $d$ and/or $k$ which have specific advantages.

\begin{theorem}\label{theofamily}
Let $q$ be a prime power, $q \geq d_0k_0 +1$.	Let $k_i \leq k_{i+1}, d_i \leq d_{i+1}$, $q^{2^{i}} \geq d_{i+1}k_{i+1}+1$,  for all $i \geq 0$. Then, the sequence $\{\mathbf{C}_{q^{2^{i}}, k_i, d_i}\}_{i \geq 0}$ is an embedding family of CFFs.
\end{theorem}
\begin{proof}
For $q \geq d_0k_0 +1$ we build a $d_0$-CFF $C_{q, k_0, d_0}$ using the polynomial construction (Construction \ref{constpoly}, Theorem \ref{smallercff}).
Since $k_i \leq k_{i+1}, d_i \leq d_{i+1}, q^{2^{i}} \geq d_{i+1}k_{i+1}+1$, for each $i \geq 0$, we apply Theorem \ref{theoinc} to embed $\mathbf{C}_{q^{2^{i}}, k_i, d_i}$ in $\mathbf{C}_{q^{2^{i+1}}, k_{i+1}, d_{i+1}}$.

\end{proof}	

The following corollaries show useful applications of Theorem \ref{theofamily}.

\begin{corollary}[Prioritizing $d$ increase]\label{dinc}
	Let $d_0\geq 1$, $k \geq 1$, and let $q$ be a prime power such that $q > d_0k$. Let $d_i = \lceil \frac{q^{2^i}}{k}\rceil - 1$, for $i\geq 1$. Then $\{\mathbf{C}_{q^{2^i}, k, d_i}\}_{i \geq 0}$ is an embedding family of CFFs. Moreover, its compression ratio is $\rho(n) = n^{1-\frac{2}{k+1}}$ and $d \sim \frac{n^{1/k+1}}{k}$.
\end{corollary}	
\begin{proof}
	We have that $d_{i}= \lceil \frac{q^{2^{i}}}{k}\rceil - 1 < \frac{q^{2^{i}}}{k}$ and therefore $d_ik <  q^{2^{i}}$, which satisfies the hypothesis of Theorem \ref{theofamily}.
 Finally, for fixed $k$ and assuming the use of all rows of the matrix, we easily calculate the compression ratio $\frac{n}{t} = \frac{(q^{2^i})^{k+1}}{(q^{2^i})^2} = \frac{n}{n^{2/k+1}} = n^{1-\frac{2}{k+1}}$, which is increasing when $k \geq 2$.
\end{proof}	

We show a few examples in Table \ref{table:k=2} and Table \ref{table:k=3} for $q = 4, 16, 256, 65536$ and for fixed values of $k$. For each $q$ and $k$ we compute $d = \lceil \frac{q}{k}\rceil - 1$ and $n = q^{k+1}$. We note that as $k$ increases, the maximum value of $d$ decreases but we get constructions with a better ratio. 


\begin{table}[]
	\caption{Example of prioritizing $d$ increases with fixed $k = 2$. }
	\centering
	\label{table:k=2}
	\begin{tabular}{r|r|r|r|r|r|r}
		\hline
		i & q & k & d & n  & t & n/t\\ \hline
		0                      & 4     & 2 & 1     & 64              & 12     &  5.33              \\
		1                      & 16    & 2 & 7     & 4096            & 240    &  17.06              \\
		2                      & 256   & 2 & 127   & 16777216        & 65280    &  257.00        \\
		3                      & 65536 & 2 & 32767 & 281474976710656 & 4294901760  &  65537.00          
	\end{tabular}
\end{table}

\begin{table}[]
	\caption{Example of prioritizing $d$ increases with fixed $k = 3$. }
	\centering
	\label{table:k=3}
	\begin{tabular}{r|r|r|r|r|r|r}
		\hline
		i & q     & k & d     & n           & t   & n/t       \\\hline
		0 & 4     & 3 & 1     & 256         & 16     &  16  \\
		1 & 16    & 3 & 5     & 65536       & 256   &  256   \\
		2 & 256   & 3 & 85    & 4294967296  & 65536  &  65536  \\
		3 & 65536 & 3 & 21845 & $65536^4$ & 4294967296 &  4294967296
	\end{tabular}
\end{table}

\begin{corollary}[Prioritizing ratio increase]\label{kinc}
		Let $d\geq 1$ and $k_0 \geq 1$, $q$ a prime power such that $q > dk_0$. Let $k_i = \lceil \frac{q^{2^i}}{d}\rceil - 1$. Then $\{\mathbf{C}_{q^{2^i}, k_i, d}\}_{i \geq 0}$ is an embedding family of CFFs. Moreover, the compression ratio is $\rho(n)= \frac{n}{\log n}$.
\end{corollary}	
\begin{proof}
		We have that $k_{i}= \lceil \frac{q^{2^{i}}}{d}\rceil - 1 < \frac{q^{2^{i}}}{d}$, and therefore $k_{i}d < q^{2^{i}}$, which satisfies the hypothesis of Theorem \ref{theofamily}.
		Finally, for fixed $d$ and assuming the use of all rows of the matrix, we easily obtain the compression ratio $\frac{n}{t} = \frac{(q^{2^i})^{k_i+1}}{(q^{2^i})^2} \leq \frac{(q^{2^i})^{\frac{q^{2^i}}{d}}}{(q^{2^i})^2} \leq \frac{n}{d \log n}$, since $d \log n = q^{2^i}\log q^{2^i} < (q^{2^i})^2$. Thus, $\frac{n}{t}$ is O($\frac{n}{\log n}$).
 \qed 
\end{proof}	

We show some examples in Table~\ref{table:d=2} and Table~\ref{table:d=3} for $q = 4, 16, 256, 65536$, fixed values of $d$, and increasing $k$. 
For each $q$ and $d$ we compute $k = \lceil \frac{q}{d}\rceil - 1$ and $n = q^{k+1}$. We note that the ratio grows very quickly as $k$ increases to its maximum.  

\begin{table}[h!]
	\caption{Example of prioritizing ratio increase with fixed $d = 2$. }
	\label{table:d=2}
	\centering
	\begin{tabular}{r|r|r|r|r|r|r}
		\hline
		i & q & k & d & n  & t & n/t\\ \hline
		0 & 4     & 1 & 2    & 16              & 12 &         1.33           \\
		1 & 16    & 7 & 2     & 4294967296            & 240  &       17895697.07           \\
		2 & 256   & 127 & 2   & $256^{128}$       & 65280   &   $2.75 \times 10^{303}$            \\
		3 & 65536 & 32767 & 2 & $65536^{32768}$ & 4294901760     &    $6.04 \times 10^{157816}$   
	\end{tabular}
\end{table}

\begin{table}[h!]
	\caption{Example of prioritizing ratio increase with fixed $d = 3$. }
	\label{table:d=3}
	\centering
	\begin{tabular}{r|r|r|r|r|r|r}
		\hline
		i & q     & k & d     & n           & t   &n/t       \\\hline
		0 & 4     & 1 & 3     & 16         & 16    & 1     \\
		1 & 16    & 5 & 3     & 16777216       & 256    & 65536    \\
		2 & 256   & 85 & 3    & $256^{86}$  & 65536  & $1.95 \times 10^{202}$    \\
		3 & 65536 & 21845 & 3 & $65536^{21846}$ & 4294967296 & $1.54 \times 10^{105211}$ 
	\end{tabular}
\end{table}

Monotone families are desirable for some applications due to their flexibility since the new tests involve only new items. By selecting specific blocks of rows for the embedding family, we are able to achieve monotone families with increasing compression ratio, which was not known in the literature \cite{hartung,nested}.

\begin{theorem}\label{monotoneinc}
	Let $d\geq 1$ and $k \geq 1$, $q$ a prime power such that $q \geq dk+1$. Let $C_{q,k,d}$ be a $d$-CFF obtained from Construction \ref{constpoly} and $\{\mathbf{C}_{q^{2^i}, k, d}\}_{i \geq 0}$ be an embedding family of $d$-CFFs for fixed $k$ and $d$, obtained from recursively applying Theorem \ref{theoinc}, where $\mathbf{C}_{q,k,d} = C_{q,k,d}$ and $\mathbf{C}_{q^{2^i}, k, d}$ be the reordered matrix as shown in Theorem \ref{theoinc}. Consider $M_{q^{2^i}, k, d}$ the submatrix of $\mathbf{C}_{q^{2^i}, k, d}$ corresponding to rows indexed by $(x_l,x_j)$ where $x_l \in \mathbb{F}_q, l = 1, \ldots, dk+1; x_j \in \mathbb{F}_{q^{2^i}}$, for all $i \geq 0$. Then $\{M_{q^{2^i}, k, d}\}_{i\geq 0}$ is a monotone family of $d$-CFF($t=(dk+1)q^{2^i}, n=q^{2^i(k+1)}$). Moreover, the compression ratio is $\rho(n) = n^{1 - \frac{1}{k+1}}$.
\end{theorem}
\begin{proof}
By fixing $x_l \in \mathbb{F}_{q},  l = 1, \ldots, dk+1$ we obtain a matrix with $dk+1$ blocks of rows and consequently $|B_f| = dk+1$, and by the same argument as in Theorem \ref{smallercff} we know that each matrix $M_{q^{2^i}, k, d}$ is a $d$-CFF with $d \leq \frac{q -1}{k}$. 
Moreover, if we look to the columns of $M_{q^{2^i}, k, d}$ that are represented by polynomials $f \in \mathbb{F}_{q^{2^{i-1}}}[x]_{\leq k}$, we know $f(x_l) = x_j \in \mathbb{F}_{q^{2^{i-1}}}$, and consequently $M_{(x_l, x_j), f} = 0$ for all the cases where $x_j \in \mathbb{F}_{q^{2^i}} \backslash \mathbb{F}_{q^{2^{i-1}}}, f \in \mathbb{F}_{q^{2^{i-1}}}[x]_{\leq k}$. It is easy to see that this matches the definition of monotone family of $d$-CFFs. Finally, if we use the maximum $d = \lfloor \frac{q-1}{k}\rfloor$ we obtain a sequence of $d$-CFF($t=q\times q^{2^i}, n=(q^{2^i})^{k+1}$), which has compression ratio $\frac{n}{t} = \frac{(q^{2^i})^{k+1}}{q\times q^{2^i}}  = \frac{n}{q n^{1/k+1}}$, which is $O(n^{1 - \frac{1}{k+1}})$.
\end{proof}

For an example of Theorem~\ref{monotoneinc} with $q=3, d=2, k=1$, we refer to Table~\ref{table:extensionField} to obtain the first two matrices in the sequence. Indeed, $M_{3, 1, 2}$ is the top left sub-matrix, and $M_{9, 1, 2}$ is the given matrix restricted to the first $dk+1 = 3$ blocks of rows (first two groups of rows in Table~\ref{table:extensionField}), and the columns corresponding to polynomials of degree up to $k=1$. The ratio of this monotone family is $\rho(n) = \sqrt{n}/3$.

\section{Using Embedding Sequences in Applications}\label{appsec}

The use of embedding families given in Theorem \ref{theofamily} requires some caution. While compression ratios are excellent when each full matrix is used, as seen in Corollary \ref{dinc}, Corollary \ref{kinc}, and Theorem \ref{monotoneinc}, bad ratios can be found when we need to add much less items than the maximum $n$ for a matrix. Note that when the number of columns of $\mathcal{M}^{(l)}$ is exceeded we need to use $\mathcal{M}^{(l+1)}$ and remove unused columns.  For example, with $q=4, d= 1, k=2$ we get a maximum $n = 64$, $t=(dk+1)q = 12$ and ratio $\rho(n) = 5.33$. If we decide to use the extension field to get larger values of $n$, the next value will be $q=16$ which gives $t=(dk+1)q = 48$. This new matrix can handle up to $n=4096$, but for the case were we only need $n = 65$ we will have a very small ratio of $\rho(n) = 1.35$. For this reason it would be desirable to develop techniques for ``smoothing out" the transition in compression ratio when we move from one matrix to the next in the embedding family.

One strategy to reduce these sharp transitions is to use values much smaller than the maximum allowed by the construction. In Figure \ref{graph} and Table \ref{graphtable} we show a choice of $q = 16, 256$,  $d = \log_4 n$, and necessary increases of $k=1, 2, 3$ to achieve the desired values of $n$. We note that as we change from one field ($q=16$) to the next one ($q=256$) we have a drop on the compression ratio, and this is due to the increase in the number of rows $(dk+1)q$ as we increase $q$. As $n$ grows, the increasing ratio is restored.

\begin{figure}[h!]
	\centering
	\includegraphics[width=1\textwidth]{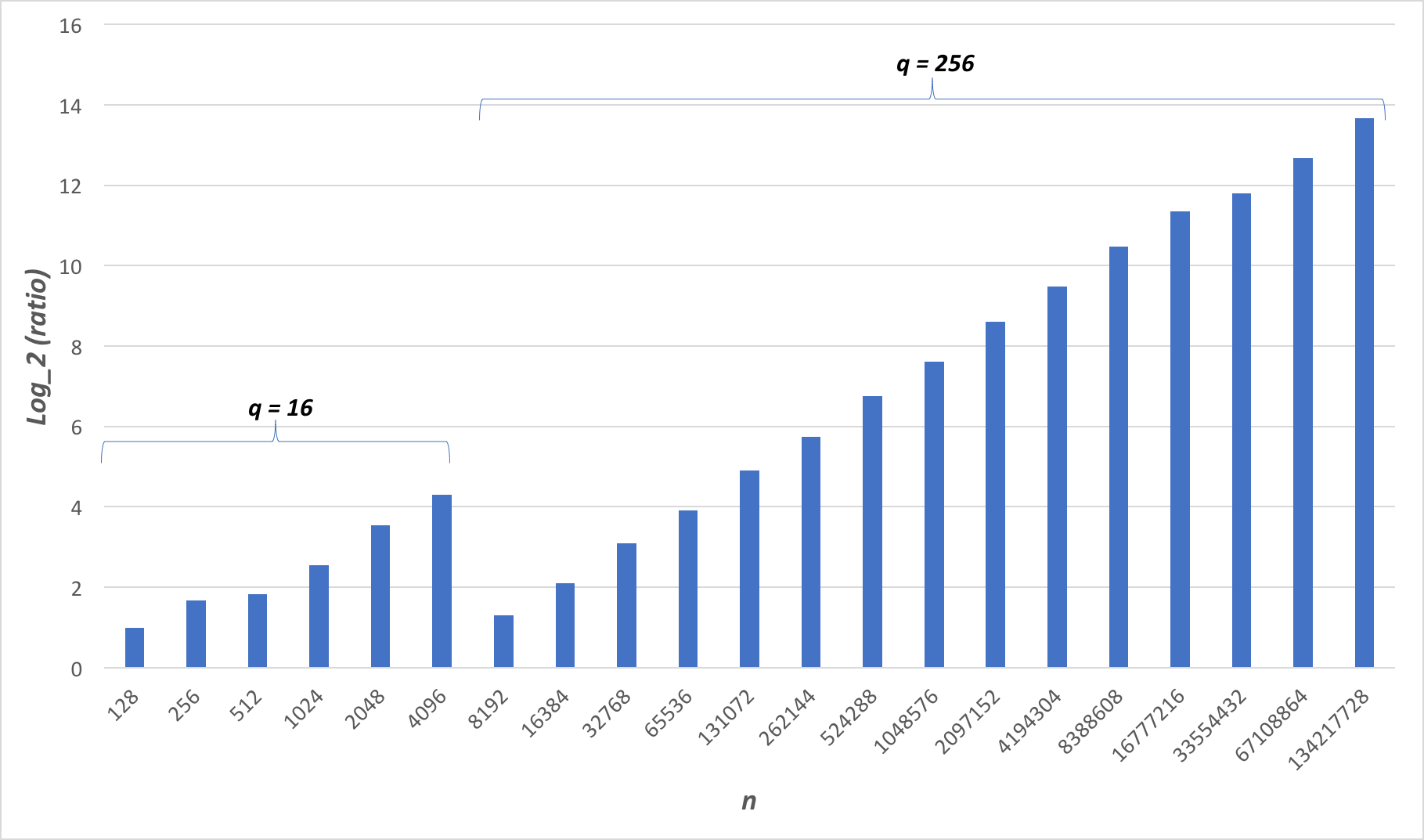}
		\caption{Compression ratio for $q=16,256; 1 \leq k \leq 3,; d=\log_4 n$.}
		\label{graph}
\end{figure}

\begin{table}[h!]
	\caption{Compression ratio for $q=16,256;1 \leq k \leq 3; d=\log_4 n$.}
	\label{graphtable}
		\centering
	\begin{tabular}{r|r|r|r|r|r}
		\hline
		q  & k & d & n   & t  & $\rho(n)=n/t$ \\ \hline
		16 & 1 & 3 & 128 & 64 & 2  \\
		16  &1	&4	&256&	80&	3.2\\
		16	&2	&4	&512&	144&	3.55\\
		16	&2&	5&	1024&	176&	5.81\\
		16	&2	&5&	2048&	176	&11.63\\
		16 &2&	6	&4096&	208	&19.69\\ \hline
		256	&2	&6	&8192	&3328	&2.46\\
		256	&2	&7	&16384	&3840	&4.26\\
		256	&2	&7	&32768	&3840	&8.53\\
		256	&2	&8	&65536	&4352	&15.05\\
		256	&2	&8	&131072	&4352	&30.11\\
		256	&2	&9	&262144	&4864	&53.89\\
		256	&2	&9	&524288	&4864	&107.78\\
		256	&2	&10	&1048576	&5376	&195.04\\
		256	&2	&10	&2097152	&5376	&390.09\\
		256	&2	&11	&4194304	&5888	&712.34\\
		256	&2	&11	&8388608	&5888	&1424.69\\
		256	&2	&12	&16777216	&6400	&2621.44\\
		256	&3	&12	&33554432&	9472	&3542.48\\
		256	&3	&13	&67108864	&10240	&6553.6\\
		256	&3	&13	&134217728	&10240	&13107.2\\
	\end{tabular}
\end{table}

\section{Generalized Construction of Embedding Families}\label{section:generalized}

 In this section, we present a generalization of the results presented in Section \ref{embeddingsec}. 
We start by defining a few objects that will be used, such as Orthogonal Arrays (OAs), Packing Arrays (PAs), and their relationships with separating hash families (SHF). Then we discuss how they can be used to construct embedding families.

\begin{definition}
	An \emph{orthogonal array} OA($v^t;t,k,v$) is an $v^t \times k$ array with elements from an alphabet of $v$ symbols, such that in any $t$ columns, every $t$-tuple of points is contained in exactly one row.
\end{definition}
\begin{definition} [Stevens and Mendelsohn\cite{PA}]
	A \emph{packing array}, PA($n;t,k,v$), is an $n\times k$ array with values from  an alphabet with $v$ symbols, such that in any $t$ columns, every $t$-tuple of points is contained in at most one row. 
\end{definition}

We note that for a PA($n;t,k,v$) we must have $n \leq v^t$, and that an OA($v^t;t,k,v$) is a packing array with maximum number of rows.

\begin{definition}[Stinson et al. \cite{Stinson2008}]
	An $(n, m, \{w_1, \ldots, w_t\})$-separating hash family is a set of functions $\mathcal{F}$, such that $|Y|=n, |X|=m,$ $f:Y \rightarrow X$ for each $f \in \mathcal{F}$, and for any sets $C_1, \ldots, C_t \subseteq \{1, 2,\ldots,n\}$ such that $|C_i|=w_i, 1 \leq i \leq t$ and $C_i \cap C_j = \emptyset$ for $i \neq j$, there exist at least one function $f \in \mathcal{F}$
	such that 
	
	\[\{f(y):y \in C_i\}\cap \{f(y):y \in C_j\}=\emptyset.\]
	
	For $|\mathcal{F}| = N$, the $(n, m, \{w_1, \ldots, w_t\})$-separating hash family is denoted \\SHF$(N; n, m, \{w_1, \ldots, w_t\})$.
\end{definition}

\begin{remark}
	A SHF$(N; n, m, \{w_1, \ldots, w_t\})$ can be depicted as an $N \times n$ matrix $A$ with entries from $\{1, 2,\ldots,m\}$, where the rows represent the functions $f$, the columns represent the elements in $Y$, and the entry in $A(f,y) = f(y)$.
	Given disjoint sets of columns $C_1, \ldots, C_t$, there exists at least one row $r$ in $A$ with the following property:
	\[\{A(r,y):y \in C_i\}\cap \{A(r,y):y \in C_j\}=\emptyset.\]  for all $i \neq j$. 
\end{remark}

The next figure shows an example of an SHF($2;6,4,\{1,2\}$), based on a construction given by Li, Van Rees and Wei \cite{monotone}.
	\begin{figure}[h]
		\centering
		\caption{An SHF($2;6,4,\{1,2\}$).}
		\label{SHF2}
		\begin{tabular}{|l|l|l|l|l|l|}
			\hline
			1&2&3&4&4&4 \\ \hline
			4&4&4&1&2&3 \\ \hline
		\end{tabular}
	\end{figure}

Now we present some relationships between packing arrays and separating hash families, and how we can use them to construct CFFs and embedding families.
In the following propositions, we consider a PA($n; t, k, v$) and the fact that any two rows have at most $t-1$ positions in common. A similar result is presented by Stinson et al. \cite{stinson2000}, where they propose the use of orthogonal arrays to construct perfect hash families, and mention that similar results can be achieved for SHFs. 


\begin{proposition}\label{shfprop}
	If $A$ is a PA($n; t, k, v$), and $w \leq \frac{k-1}{t-1}$, then $A^T$ is a SHF($N=k; n, m = v, \{1,w\}$).
\end{proposition}
\begin{proof}
	Take $w+1$ rows $r, r_1, r_2, \ldots, r_w$ of $A$. We need to find a column $j$ such that
	\begin{equation}\label{property:PA}
	A[r_l,j] \neq A[r,j] \text{ for all } l \in \{1, \ldots, w\}.
	\end{equation}
	Let $B_i = \{A[r,l] = A[r_i,l]: l \in \{1, \ldots, k\} \}$, $i \in \{1, \ldots, w\}$, we know $|B_i| \leq t-1$ since $A$ is a PA. We claim there exists a column $j \in T = \{1, \ldots, k\} \backslash \bigcup_{i=1}^{w}B_i$, which is the desired column in (\ref{property:PA}). In fact, 
	
	\[|T| = |\{1, \ldots, k\}\backslash \bigcup_{i=1}^{w}B_i| \geq k - \sum_{i=1}^{w}|B_i| \geq  k - w(t-1) \geq k-\frac{(k-1)}{(t-1)}(t-1) = 1\]
	

\noindent	
	so column $j$ exists, which becomes row $j$ when using $A^T$, and therefore guarantees the necessary and sufficient property for a SHF of type $\{1,w\}$.  
\end{proof}

Now we show we can construct $d$-CFF incidence matrices from SHFs. Similar results can be found in \cite{mastersKim} from perfect hash families, and a more general result can be found in \cite{StinsonWei2004} for the construction of $(w,d)$-CFFs from SHFs of type $\{w,d\}$. 


\begin{construction}\label{constSHFCFF}
	Let $A$ be a $SHF(N; n, m,\{1,w\})$, we build a $mN \times n$ matrix $\mathcal{M}$ as follows. Index each row of $\mathcal{M}$ with tuples $(i,x)$, for $x = 1, \ldots, m$, $i=1, \ldots, N$, and each column $j = 1, \ldots, n$, then
	
	\[\mathcal{M}_{(i,x),j} = \begin{cases}
	1 & \text{ if } A_{i,j} = x, \\
	0 & \text{ otherwise.}
	\end{cases}\]
	
\end{construction}

\begin{proposition}\label{SHFCFF}
	Let A be a SHF($N; n, m,\{1,w\}$), then $\mathcal{M}$ built via Construction \ref{constSHFCFF} gives a $d$-CFF($mN, n$), with $d \leq w$. 
\end{proposition}
\begin{proof}
If we take $w+1$ columns $c_0, c_1, \ldots, c_w$ of a SHF $A$ of type $\{1,w\}$, we know there will be a row $i$ such that $A[i,c_0] \neq A[i,c_j]$ for $j \in \{1, \ldots, w\}$. In Construction \ref{constSHFCFF} we convert each element $x$ in $A$ into a vector of size $m$ with one ``1" in position $x$ and ``0" in the remaining positions. Due to the SHF property, we note that there is at least one row $(i,x)$ such that $\mathcal{M}_{(i,x),c_0}= 1$ while $\mathcal{M}_{(i,x),c_j} = 0,$  $(1 \leq j \leq w)$, for any columns $c_0, c_1, \ldots, c_w$ in $\mathcal{M}$, which matches the requirements for a $w$-CFF. 
Moreover, since we expand each row of $A$ into an array of size $m$, $\mathcal{M}$ is $w$-CFF($mN, n$).


\end{proof}	


\begin{remark}
	If we use a PA($n; t, k, v$) to build a SHF($N=k; n, m = v, \{1,w\}$) with $w \leq \frac{k-1}{t-1}$ as in Proposition \ref{shfprop} and then apply Construction \ref{constSHFCFF}, we obtain a $d$-CFF($mN, n$) with $d \leq \frac{k-1}{t-1}$.
\end{remark} 

\begin{remark}
	For the special case where we use an OA($q^t;t,q,q$) constructed using polynomials over finite fields using Bush's construction~\cite{bush1952}, $q$ a prime power, to build a SHF($q; q^t, q,\{1,w\}$) and then apply Construction \ref{constSHFCFF}, we have a $d$-CFF($q^2, q^t$) for $ d \leq \frac{q-1}{t-1}$, which is equivalent to Construction \ref{constpoly} using polynomials. 
\end{remark} 

When we construct a $d$-CFF from a SHF that came from a PA, we observe a similar property of blocks of rows giving increasing values of $d$ as in Theorem \ref{smallercff}.
\begin{proposition}
	Let $P$ be a PA($n;t,k,v$), $A = P^T$ be a SHF($N=k; n, m = v, \{1,w\}$) with $w \leq \frac{k-1}{t-1}$, and $k_i =  i(t-1) + 1$. Consider $\mathcal{M}$ the $w$-CFF($k \times v, n$) obtained from Construction \ref{constSHFCFF} using $A$ and let a ``block" of rows be any consecutive $m$ rows indexed by $(i,1), \ldots, (i,m)$. When we restrict $\mathcal{M}$ to any $k_i$ blocks of rows we obtain a $i$-CFF($k_i \times v, n$).
\end{proposition}
\begin{proof}
	From Proposition \ref{shfprop} we know that a SHF of type $\{1,w\}$ can be created from a PA $P$ of strength $t$ as long as the number of columns of the PA is at least $k \geq w(t-1) + 1$. We can restrict the packing array $P$ to $k_i=i(t-1) + 1$ columns, $1 \leq i \leq w$, without compromising its properties and therefore obtain a SHF $A_i$ of type $\{1,i\}$. By applying Construction \ref{constSHFCFF} with $A_i$ we obtain a $i$-CFF($k_i \times v, n$). 
\end{proof}

The next proposition shows that a special sequence of PAs generalizes the polynomial construction of embedding family given in Theorem \ref{theofamily}.

\begin{proposition}\label{embeddingPA}
	Let $(\mathcal{P}^{(l)})_l$ be a sequence of PAs, where $\mathcal{P}^{(l)}$ is a PA($n_l; t_l, k_l, v_l$), $n_l \leq n_{l+1}, t_l \leq t_{l+1}, k_l \leq k_{l+1}, v_l \leq v_{l+1}$, and
	\[\mathcal{P}^{(l+1)}= \begin{pmatrix} \mathcal{P}^{(l)} & Y\\ Z & W \end{pmatrix},\]
	
	\noindent
and in addition $d_l \leq \frac{k_l-1}{t_l-1}$, $d_l \leq d_{l+1}$, for all $l \geq 1$. Then there exists an embedding family of $d_l$-CFF($k_l \times v_l, n_l$).
	
\end{proposition}
\begin{proof}
Let $\mathcal{A}^{(l)}$ be obtained by transposing $\mathcal{P}^{(l)}$. By Proposition \ref{shfprop}, since $d_l \leq \frac{k_l-1}{t_l-1}$, $\mathcal{A}^{(l)}$ is a SHF($N_l = k_l; n_l, m_l = v_l, \{1, d_l\}$), and consequently $(\mathcal{A}^{(l)})_l$ is a sequence of SHFs with the following form
	
	
	\[\mathcal{A}^{(l+1)}= \begin{pmatrix} \mathcal{A}^{(l)} & Z^T\\ Y^T & W^T \end{pmatrix}.\]
	Then, for each $\mathcal{A}^{(l)}$ we apply Construction \ref{constSHFCFF} and we get $\mathcal{M}^{(l)}$, a $d_l$-CFF($k_l \times v_l, n_l$) by Proposition~\ref{SHFCFF}. It is easy to see that the smaller CFF $\mathcal{M}^{(l)}$ is in the top corner of $\mathcal{M}^{(l+1)}$, and all the requirements for being an embedding family are satisfied. 
	
\end{proof}

\section{Conclusion}\label{sec:conclusion}
This paper introduces the idea of embedding families of CFFs as a general framework to look at how CFF constructions can be leveraged to optimize parameters of interest to applications. The infinite families obtained in Section \ref{embeddingsec} have excellent asymptotic compression ratios, some matching the information theoretical upper bound,
and permit increase of $d$ and $n$.
However, these constructions  present abrupt increases of $t$ and $n$ (when moving to the next $q$) that need to be ``smoothed out" for  improved use in applications, as discussed in Section \ref{appsec}. An important direction for future work is the study of  adequate growth for $d$ and $k$ to yield smoother instances of these families.\\
%
%
%
%

\bibliographystyle{plain}
\bibliography{bibliografia}

%
%





\end{document}